\documentclass[a4paper,11pt]{article}
\usepackage{times,amsfonts,amssymb,amsmath,algorithm,algorithmicx,algpseudocode}
\usepackage{graphics,graphicx}
\usepackage[bitstream-charter]{mathdesign}
\usepackage[T1]{fontenc}
	\usepackage{amssymb, amsmath, amsthm}
	\usepackage{pstricks}
	\usepackage{enumerate}
	\usepackage{pst-plot}
	\usepackage{graphicx}
	\usepackage{subfig}
	\usepackage[margin=2.7cm]{geometry}
	
	\newtheorem{thm}{Theorem}[section]
	\newtheorem{lem}[thm]{Lemma}
	\newtheorem{prop}[thm]{Proposition}
	\newtheorem{cor}[thm]{Corollary}
	\newtheorem{conj}[thm]{Conjecture}
	\newtheorem*{totalconj}{Total Coloring Conjecture'}
\newtheorem{thmB1}{Theorem}
\newtheorem{thmB2}[thmB1]{Theorem}

	\newtheorem*{conj1}{Conjecture 1}
	
	\newtheorem{rem}[thm]{Remark}
	
	\newcommand{\Gmn}{G^{\frac{m}{n}}}
	
	\newcommand{\Got}{G^{\frac{1}{2}}}
	\newcommand{\Gon}{G^{\frac{1}{n}}}
	
	\newcommand{\N}{\mathbb{N}}

\begin{document}
	
	\title{On coloring of fractional powers of graphs}
	\author{
	Stephen Hartke
	\thanks{Department of Mathematics,
	University of Nebraska Lincoln, Lincoln, NE 68588-0130, USA {\tt hartke@math.unl.edu}.}
	\and
	Hong Liu\thanks{Department of Mathematics,
	University of Illinois at Urbana-Champaign, Urbana, Illinois 61801, USA {\tt hliu36@illinois.edu}.}
	\and
	\v{S}\'{a}rka Pet\v{r}\'{i}\v{c}kov\'{a}
	\thanks{Department of Mathematics,
	University of West Bohemia, Plzen 30614, CZ {\tt sarpet@kma.zcu.cz}.}
	}
	\maketitle
	
	\begin{abstract}
	For $m, n\in \N$, the fractional power $\Gmn$ of a graph $G$ is the $m$th power of the $n$-subdivision of $G$, where the $n$-subdivision is obtained by replacing each edge in $G$ with a path of length $n$. It was conjectured by Iradmusa that if $G$ is a connected graph with $\Delta(G)\ge 3$ and $1<m<n$, then $\chi(\Gmn)=\omega(\Gmn)$. Here we show that the conjecture does not hold in full generality by presenting a graph $H$ for which $\chi(H^{\frac{3}{5}})>\omega(H^{\frac{3}{5}})$. However, we prove that the conjecture is true if $m$ is even. We also study the case when $m$ is odd, obtaining a general upper bound $\chi(\Gmn)\leq \omega(\Gmn)+2$ for graphs with $\Delta(G)\geq 4$.	
	\end{abstract}
	
	
	\textbf{Math. Subj. Class.}: 05C15
	\section{Introduction}\label{S_Introduction}
	Let $G$ be a simple finite graph, and let $m$ and $n$ be positive integers. The \emph{$n$-subdivision} of $G$, denoted by $G^{\frac{1}{n}}$,
	is the graph formed from $G$ by replacing each edge with a path of length $n$. The \emph{$m$-power} of $G$, denoted by $G^m$, is the graph constructed from
	$G$ by joining every two distinct vertices with distance at most $m$. The \emph{fractional power} $\Gmn$ is then defined to be the $m$-power of the $n$-subdivision of $G$;
	that is, $\Gmn=(G^{\frac{1}{n}})^m$. Here we study the relation of $\chi(\Gmn)$ and $\omega(\Gmn)$.

	A \emph{total coloring} of $G$ is a coloring of its vertices and edges such that no adjacent vertices, no adjacent edges, and no incident edge and vertex have the same color.
	 The \emph{total chromatic number} $\chi''(G)$ of $G$ is the least number of colors in such a coloring. The famous Total Coloring Conjecture, formulated independently by Behzad~\cite{B} and Vizing~\cite{V}, states that $\chi''(G)\le\Delta(G)+2$ for any simple graph $G$. Since $\chi''(G)=\chi(G^{\frac{2}{2}})$, and $\omega(G^{\frac{2}{2}})=\Delta(G)+1$ for every graph
	$G$ with $\Delta(G)\ge 2$, we can rewrite the Total Coloring Conjecture as follows:
	
	\begin{totalconj}
	[\cite{B,V}] If $G$ is a simple graph with maximum degree at least 2, then $\chi(G^{\frac{2}{2}})\le\omega(G^{\frac{2}{2}})+1$.
	\end{totalconj}
	
	If $m<n$, then maximum cliques of $G$ are somewhat separated, so less colors may be needed, as conjectured by Iradumusa~\cite{I}.
	\begin{conj1}[\cite{I}]\label{conj1}
	 If $G$ is a connected graph with $\Delta(G)\ge 3$ and $1<m<n$, then $\chi(\Gmn)=\omega(\Gmn)$.
	\end{conj1}
	Iradmusa showed that Conjecture~\ref{conj1} is true for $m=2$. We refer the reader to \cite{I} for the proof and several other results for $m>2$.
	
	The purpose of this paper is to further investigate Conjecture~\ref{conj1}. Our first results is that Conjecture~\ref{conj1} is true if $m$ is even.
\begin{thmB1}\label{thmB1}
	 If $G$ is a graph with $\Delta(G)\geq 3$ and $1<m<n$ with $m$ even, then $\chi(\Gmn)=\omega(\Gmn)$.
	\end{thmB1}
We also study the conjecture when $m$ is odd. We show that the conjecture does not hold in full generality. In particular, it is not true for the cartesian product $C_3\square K_2$  of $C_3$ and $K_2$ (triangular prism), $m=3$, and $n=5$. However, we give the following general bound.
\begin{thmB2}\label{thmB2}
	 If $G$ is a graph with $\Delta(G)\geq 4$ and $1<m<n$ with $m$ odd, then $\chi(\Gmn)\leq\omega(\Gmn)+2$.
	\end{thmB2}

	 The paper is organized as follows. In Section~\ref{S_Notation}, we introduce notation and some known results that are used later in the paper.
The proof of Theorem~1 is given in Section~\ref{S_Even} ($\Delta(G)\geq 4$) and Section~\ref{S_Cubic} ($\Delta(G)=3$).
In Section~\ref{S_Odd} we prove Theorem~2 for graphs that are not complete, and we show that  Conjecture~\ref{conj1} holds for infinitely many $n$ (about ``half'' of the values) if $m$ is odd.
Section~\ref{S_Complete} deals with complete graphs, where it is shown Conjecture~\ref{conj1} is true for all complete graphs. Finally, Section~\ref{S_dynamic} discusses the connection between $r$-dynamic coloring and Conjecture~\ref{conj1}.
		
We make the following conjecture.
	\begin{conj}\label{ourconj}
	Conjecture~\ref{conj1} holds except when $G=C_3\square K_2$.
	\end{conj}
	
	\section{Notation and Preliminaries}\label{S_Notation}
	
	In this paper we only consider finite simple graphs. Next we always assume that $m$ and $n$ are positive integers such that $m<n$.	
The graph $\Gmn$ is constructed from $G$ in two steps. First, every edge $uv$ is replaced by a path $P_{uv}$ (called a \emph{superedge}) of length $n$, forming $\Gon$. Second, edges joining vertices of distance at most $m$ in $\Gon$ are added, forming $\Gmn$.
	
	 Let $(uv)_i$ ($i=0,\dots,n$) be a vertex of $\Gmn$ that lies on $P_{uv}$ and has distance $i$ from $u$ in $\Gon$. If $i=0$ or $i=n$, then $(uv)_i$ is a \emph{branch vertex}, otherwise $(uv)_i$ is an \emph{internal vertex}.
	
	For an edge $uv$ in $G$, the ordered $\lfloor \frac{m}{2}\rfloor$-tupple of vertices of $V(\Gmn)$ defined by $$B_{uv}=((uv)_1,\dots, (uv)_{\lfloor \frac{m}{2}\rfloor})$$ is called a \emph{bubble (at} $u$). If $m$ is odd, then we say that the set of vertices $$C_{u}=\{(uv)_{(m+1)/2}\in V(\Gmn):uv\in E(G)\}$$ is the \emph{crust at} $u$.
	Lastly, $M_{uv}$, called a \emph{middle part}, is the tuple of vertices between the two bubbles (or two crusts if $m$ is odd) on the edge $uv$ defined by
	$$M_{uv}=((uv)_{\lceil \frac{m}{2}\rceil+1},\dots, (uv)_{n-(\lceil \frac{m}{2}\rceil+1)}).$$
	
	\begin{figure}
  \begin{center}
  \subfloat[$m$ even.]{\label{def_even}
    \includegraphics{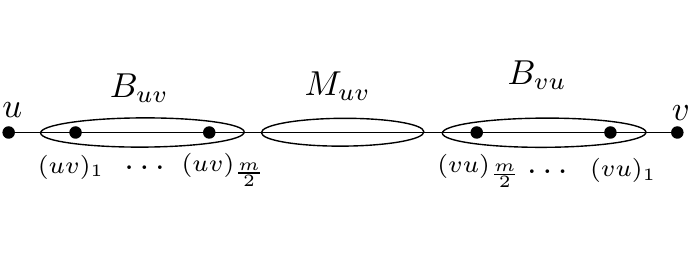}}
     \hspace{0.5cm}
      \subfloat[$m$ odd.]{\label{def_odd}
    \includegraphics{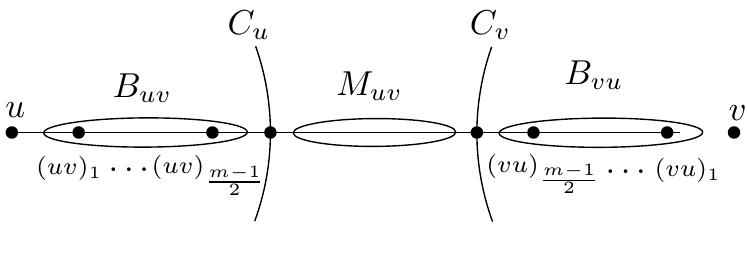}}
      \caption{Bubbles, crusts, and middle parts.}\label{def}
  \end{center}
\end{figure}

	For a $k$-tuple $A=(a_1,\dots,a_k)$ we define $\overline{A}=(a_k,\dots,a_1)$. Also, we write $A[i:j]=(a_i,\dots,a_j)$ and $A[i]=(a_i)$. So, we have $\overline{A}[i:j]=(a_{k-i+1},\dots,a_{k-j+1})$ and $\overline{A}[i]=(a_{k-i+1})$.
If $B=(b_1,\dots,b_k)$ is another $k$-tuple, then $\varphi(A)=B$ means $\varphi(A[i])=B[i]$ (i.e. $a_i=b_i$) for all $i\in[k]$ for some function $\varphi$. To shorten notation, we often write $[k]$ instead of $\{1,\dots,k\}$. Lastly, we use the symbol $*$ to denote the merging of two tuples as follows: $(a_1,\dots, a_i)*(a_{i+1},\dots,a_n):=(a_1,\dots ,a_n)$.

We are ready to state results of Iradmusa that we use later in this paper. The first theorem gives the exact values of $\omega(\Gmn)$. 	
	\begin{thm}[\cite{I}]\label{omegaB}
	 If $G$ is a graph and $n, m\in \N$ such that $m<n$, then
	$$\omega(\Gmn)=\left\{
	\begin{array}{llllllllllll}
	 m+1                    & \mbox{ if } \Delta(G)=1,\\
	 \frac{m}{2}\Delta(G)+1 & \mbox{ if }\Delta(G)\geq 2, m \mbox{ even},\\
	 \frac{m-1}{2}\Delta(G)+2 & \mbox{ if }\Delta(G)\geq 2, m \mbox{ odd}.
	\end{array}
	\right.
	$$
	\end{thm}
		Since $\chi(H)\geq \omega(H)$ for any graph $H$, it suffices to show $\chi(\Gmn)\leq \omega(\Gmn)$ to prove Conjecture~\ref{conj1}. By Theorem~\ref{omegaB}, we thus need to construct a coloring of $\Gmn$ using $\frac{m}{2}\Delta(G)+1$ colors if $m$ is even, and
	$\frac{m-1}{2}\Delta(G)+2$ colors if $m$ is odd.

By \cite[Lemma1]{I}, if $\chi(\Gmn)=\omega(\Gmn)$, then $\chi(G^{\frac{m}{n+m+1}})=\omega(G^{\frac{m}{n+m+1}})$.
This result can be generalized as follows.
	
	\begin{lem}\label{threeB}
	 Let $G$ be a graph, and $m,n\in\N$ such that $m<n$. If $\chi(\Gmn)\leq\omega(\Gmn)+c$ for some $c\in \N\cup \{0\}$, then $\chi(G^{\frac{m}{n+m+1}})\leq\omega(G^{\frac{m}{n+m+1}})+c$.
	\end{lem}
	
	\begin{proof}
By Theorem~\ref{omegaB}, $\omega(\Gmn)=\omega(G^{\frac{m}{n+m+1}})$. By the proof of Lemma~1 in \cite{I}, it holds that $\chi(G^{\frac{m}{n+m+1}})\leq \chi(\Gmn)$. Therefore
	$$\chi(G^{\frac{m}{n+m+1}})\leq \chi(\Gmn)\leq \omega(\Gmn)+c=\omega(G^{\frac{m}{n+m+1}})+c.$$
	\end{proof}
	\begin{thm}[\cite{I}]\label{one}
	 If $G$ is a connected graph with $\Delta(G)\geq 3$ and $m\in \N$, then  $\chi(G^{\frac{m}{m+1}})=\omega(G^{\frac{m}{m+1}})$.
	\end{thm}
	
The next lemma follows by inductively applying Lemma~\ref{threeB} and Theorem~\ref{one}. We will use it repeatedly.
	
	\begin{lem}\label{range}
If $\chi(\Gmn)\leq\omega(\Gmn)+c$ for all $n=m+2, \dots, 2m+1$ and some $c\in \N\cup\{0\}$, then $\chi(\Gmn)\leq\omega(\Gmn)+c$ for all $n$ with $n>m$. In Particular, if Conjecture~\ref{conj1} holds for $n=m+2, \dots, 2m+1$, then it holds for all $n$ with $n>m$.
	\end{lem}
	
	
Throughout this paper, we will also use the following well-known result by K\"{o}nig~\cite{K}.
	
	\begin{lem}[\cite{K}]\label{regular}
	For every graph $G$ with maximum degree $\Delta$ there exists a $\Delta$-regular graph containing $G$ as an induced subgraph.
	\end{lem}
	
	Lemma~\ref{regular} enables us to restrict our attention to regular graphs. Indeed, if $G$ is not regular, we can find a $\Delta$-regular graph $H$, color $H$, and use the coloring of $H$ on $G$. The following result is proven in \cite[Theorem 3]{I} using a special type of vertex ordering and induction. Here we give another simple proof.
	
	\begin{lem}\label{G23}
	 If $G$ is a connected graph with maximum degree $\Delta\geq 3$, then there exists a proper coloring $h:V(G^{\frac{2}{3}})\rightarrow \{0,\dots,\Delta\}$ using the color $0$ exactly on branch vertices.
	\end{lem}
	
	\begin{proof}
First, we color all branch vertices with $0$. It remains to color the internal vertices with colors $1,\dots, \Delta$. Let $H$ be a graph that arises if we delete all branch vertices of $G^{\frac{2}{3}}$. Then $\Delta(H)=\Delta$, and $H$ is neither an odd cycle nor a complete graph. Thus, $\chi(H)\leq \Delta$ by Brooks' theorem.
	\end{proof}
	
Observe that the graph $H$ that arises if we delete all branch vertices of $G^{\frac{2}{3}}$ is isomorphic to the line graph of $\Got$. Therefore, we can think of coloring of internal vertices of $G^{\frac{2}{3}}$ as of coloring of edges of $\Got$. So, Lemma~\ref{G23} equivalently states that for a graph with degree $\Delta\geq 3$, there is a proper edge coloring $h:E(G^{\frac{1}{2}})\rightarrow \{1,\dots,\Delta\}$. We refer to such a coloring as a \emph{half-edge coloring of $G$}, and to the 'halves' of edges of $G$ that correspond to the edges in $G^{\frac{1}{2}}$ simply as \emph{half-edges of $G$}. For an edge $e=uv$ in $G$, $e_{uv}$ denotes the half-edge on $e$ that is adjacent to $u$.

	\begin{lem}\label{L_half}
	 If $G$ is a graph with maximum degree $\Delta\geq 3$, then there exists a (proper) half-edge coloring  $h:E(G^{\frac{1}{2}})\rightarrow [\Delta]$.
	\end{lem}
	
	\section{Theorem~1 for non-cubic graphs}\label{S_Even}
	
In this section we show that Conjecture~\ref{conj1} is true if $m$ is even and $\Delta(G)\geq 4$. We sketch the basic idea of the technique used here. Suppose that we have a half-edge coloring of $G$ using $\Delta(G)$ colors. For each color $a$ we introduce $\frac{m}{2}$ new colors $a_1,\dots, a_{\frac{m}{2}}$, and use them on each bubble whose corresponding half-edge has color $a$. So, we use $\frac{m}{2} \Delta(G)$ colors to color vertices of bubbles. As a next step we show that it is possible to color the middle vertices from the same set of colors. So, if we use an additional color $0$ for branch vertices of $\Gmn$, then we can conclude that $\chi(\Gmn)\leq \frac{m}{2} \Delta(G)+1=\omega(\Gmn)$, as needed.

	\begin{thm}~\label{miseven}
	 If $m$ is even and $G$ is a connected graph with $\Delta(G)\ge 4$, then $\chi(\Gmn)=\omega(\Gmn)$.
	\end{thm}
	
	\begin{proof}
	By Lemma~\ref{regular}, we can assume that $G$ is $\Delta$-regular graph, where $\Delta=\Delta(G)$. Also, it is sufficient to prove the claim for $m+2\le n\le 2m+1$ (Lemma~\ref{range}).
		By Theorem~\ref{omegaB}, $\omega(\Gmn)=\frac{m}{2} \Delta+1$. Since trivially $\chi(\Gmn)\geq \omega(\Gmn)$, we only need to construct a proper vertex coloring $\Gmn$ that uses $\frac{m}{2} \Delta+1$ colors.

Let $h: E(\Got) \rightarrow \{1,\dots,\Delta\}$, where $1=(1_1,\dots,1_{ \frac{m}{2}}), \dots, \allowbreak\Delta=(\Delta_1,\dots,\Delta_{\frac{m}{2}})$, be a half-edge coloring of $G$ whose existence is ensured by Lemma~\ref{L_half}. Then we define a vertex coloring $\varphi:\nobreak V(\Gmn)\rightarrow \{0,1_1,\dots,1_{\frac{m}{2}},\allowbreak2_1\dots,2_{\frac{m}{2}},\allowbreak\dots,\Delta_{1},\dots,\Delta_{\frac{m}{2}}\}$ as follows.	
	
	\begin{enumerate}
	 \item \textbf{branch vertices:} $\varphi(v)=0$ for every branch vertex $v$.
	  \item \textbf{bubbles:} $\varphi(B_{uv})=h(e_{uv})$. (Recall that this means $\varphi(B_{uv}[i])=h(e_{uv})[i] \ \forall i\in [\frac{m}{2}]$.)
	   \item \textbf{middle parts:} Consider a superedge $P_{uv}$.
	   Since $\Delta(G)\geq 4$, there exist two color-tuples $a$, $b \in [\Delta]\setminus \{h(e_{uv}),h(e_{vu})\}$. Let $k=\lceil\frac{1}{2}|M_{uv}|\rceil$ and $l=\lfloor\frac{1}{2}|M_{uv}|\rfloor$. Then we define  $\varphi(M_{uv}[1:k])=\overline{a}[1:k]$ and $\varphi(M_{vu}[1:l])=\overline{b}[1:l]$.
	   	  See Figure~\ref{even}.
	\end{enumerate}
						
	\begin{figure}
  \begin{center}
  \subfloat[Coloring $h$.]{\label{even1}
    \includegraphics{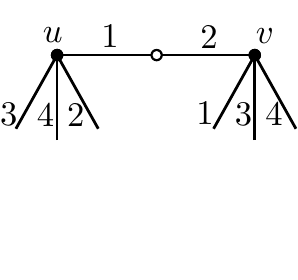}}
     \hspace{1cm}
      \subfloat[The corresponding coloring $\varphi$ of $\Gmn$.]{\label{even2}
    \includegraphics{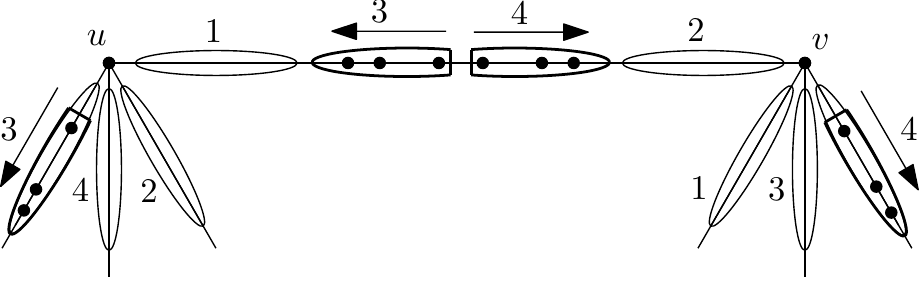}}
      \caption{Using half-edge coloring of $G$ to color $\Gmn$.}\label{even}
  \end{center}
\end{figure}	
	
	We show that the coloring $\varphi$ is proper. We need to check that the distance between any two vertices $x$, $y$ with the same color is greater then $m$ in the subgraph $H:=G^{\frac{1}{n}}$ of $\Gmn$. This is true if $\varphi(x)=\varphi(y)=0$ since then $x$ and $y$ are branch vertices. Let now $\varphi(x)=\varphi(y)\neq 0$.
	
\begin{enumerate}
	 	 \item $x\in B_{uv}$ and $y \in B_{u'v'}$: First, $u v\neq u' v'$ since the colors of vertices of each bubble are pairwise distinct.
	  Next, neither $u\neq u'$ nor $v\neq v'$ since $h$ is proper and thus $B_{uv}$ and $B_{u'v'}$ cannot be at the same vertex. Also, $B_{uv}$ and $B_{u'v'}$ are not on the same superedge, so either $u\neq v'$ or $v\neq u'$. If $u\neq v'$ and $v\neq u'$, then $uv$ and $u'v'$ are two vertex-disjoint edges in $G$ and so $d_{H}(x,y)\geq n+2>m$. If $u\neq v'$ and $v=u'$, then $d_{H}(x,y)=d_{H}(u,v)=n>m$ since $d_{H}(u,x)=d_{H}(u',y)$ by definition of $\varphi$. The case $u= v'$ and $v\neq u'$ is analogous.
	
	 \item $x\in M_{uv}$ and $y \in M_{u'v'}$: As in the previous case, $uv\neq u'v'$. But then $d_{H}(x,y)$ since any two vertices from two different middle parts are in distance at least $m+2$ in $\Gon$.
	
	 \item $x\in B_{uv}$ and $y\in M_{u'v'}$: If $uv$ and $u'v'$ are vertex-disjoint edges in $G$, then again $d_H(x,y)\ge n+2>m$. So we may assume that $u=u'$. Then $v\neq v'$ since $B_{uv}$ and $M_{u'v'}$ do not belong to the same superedge by definition of $\varphi$. Let $i$ be such that $y=M_{uv}[i]$. Since $\varphi(x)=\varphi(y)$, we have by definition of $\varphi$ that $x=B_{uv}[\frac{m}{2}-i+1]$. It follows that $d_H(x,u)=\frac{m}{2}-i+1$ and $d_H(u',y)=\frac{m}{2}+i$, which together gives $d_H(x,y)=\frac{m}{2}-i+1+\frac{m}{2}+i=m+1$.
\end{enumerate}	
	\end{proof}
		
\section{Theorem~1 for cubic graphs} \label{S_Cubic}

 For a cubic graph $G$, let $h:E(\Got)\rightarrow [3]$ be a proper half-edge coloring. A cycle in $G$ that uses only two of the colors $1,2,3$ on its half-edges is called a \emph{bad cycle} in $(G,h)$. The coloring $h$ is called a \emph{good half-edge coloring} if there is no bad cycle in $(G,h)$.
	
	\begin{lem}\label{lem1}
There exists a good half-edge coloring for every cubic graph $G$.
	\end{lem}
	\begin{proof}
	We consider an arbitrary half-edge coloring $h: E(\Gon)\rightarrow [3]$ of $G$. Let $C_1,\dots, C_k$ be bad cycles of $(G,h)$. Observe that the cycles $C_1, \dots, C_k$ are pairwise edge-disjoint.
		We prove that we can eliminate the bad cycles $C_1,\dots, C_k$ by induction on $k$. If $k=0$, then there is nothing to prove. So suppose that $k\geq 1$.

	Consider the cycle $C_k$ with its half-edges colored with colors $a,b\in[3]$. Then every half-edge incident with a vertex of $C_k$ has color $c\in [3]\setminus\{a,b\}$. 	
	For an arbitrary vertex $u$ of $C_k$, let $v$ be the neighbor of $u$ such that $v\not\in V(C_k)$. So we have $h(e_{uv})=c$. Let $w$ be the neighbor of $u$ in $C_k$ such that $h(e_{uw})\neq h(e_{vu})$. Then we define a half-edge coloring $h'$ of $G$ by $h'(e_{uv})=h(e_{uw})$, $h'(e_{uw})=h(e_{uv})$, and $h'(e_{xy})=h(e_{xy})$ for the remaining half-edges $e_{xy}$ of $G$ (see Figure~\ref{cubic1}). Now $C_k$ uses all the three colors $1,2,3$ and the new coloring $h'$ is again a proper half-edge coloring.

	We show that no new bad cycle arises. Suppose it does. Let $C_k'$ be a bad cycle in $(G,h')$ that was not a bad cycle in $(G,h)$. Then either the edge $uv$ or the edge $uw$ is in $C_k'$. Suppose $uv\in E(C_k')$ first. Then $C_k'$ uses colors $a$ and $b$ on its half-edges, and thus contains the path $P_k=C_k\setminus\{uw\}$. But the vertex $w$ of $P_k$ is adjacent to only one edge that uses colors $a$ and $b$ on its half-edges (since one of the colors $a,b$ was changed to $c$ on $uw$), a contradiction. The case when $uw\in E(C_k')$ also leads to a contradiction since no other edge than $uw$ incident with $u$ has color $c$ on its half-edges.
	\end{proof}

\begin{figure}
  \begin{center}
  \subfloat[Switching the colors on $e_{uv}$ and $e_{uw}$.]{\label{cubic1}
    \includegraphics{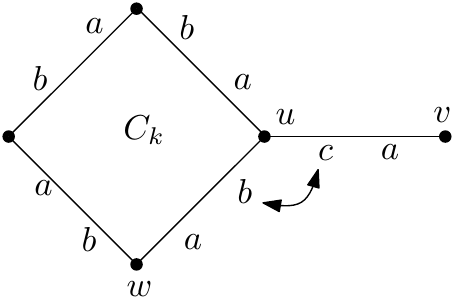}}
     \hspace{2cm}
      \subfloat[Orienting $c$-half-edges.]{\label{cubic2}
    \includegraphics{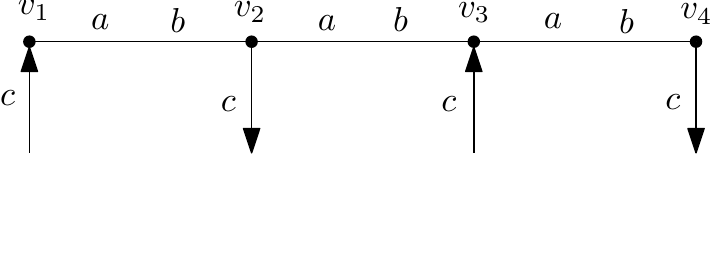}}
      \caption{}\label{cubic}
  \end{center}
\end{figure}
	
For a half-edge coloring $h$, an \emph{$a$-half-edge} is a half-edge colored with the color $a$ by $h$, and an \emph{$ab$-edge} is an edge with the colors $a,b$ used on its half-edges.
	
An \emph{orientation} $\overrightarrow{G}$ of a graph $G$ is obtained by assigning a direction to each edge. We say that a half-edge $e_{uv}$ is \emph{oriented inwards} in if the head of the corresponding edge in the orientation of $\Got$ is a branch vertex of $G$. Otherwise we say that $e_{uv}$ is \emph{oriented outwards}.

	\begin{lem}\label{lem2}
 For every good half-edge coloring $h$ of $G$ there exists an orientation of $\Got$ such that for every edge $uv\in E(G)$, if $u'\in N_G(u)\setminus \{v\}$ and $v'\in N_G(v)\setminus \{u\}$ with $h(e_{uu'})=h(e_{vv'})$, then one of the half-edges $e_{uu'}$, $e_{vv'}$ is oriented inwards and the other one outwards.
	\end{lem}
	
	\begin{proof}
For $a,b\in[3]$, $a\neq b$, let $G_{ab}$ be the graph induced by $ab$-edges of $G$. Since there are no bad cycles in $(G,h)$, the graph $G_{ab}$ is a disjoint union of paths $P_1,\dots, P_l$. Every half-edge adjacent to a vertex of $G_{ab}$ has color $c \in [3]\setminus \{a,b\}$. We orient all these neighboring $c$-half-edges subsequently for each $P_1,\dots, P_l$. For $P_i=v_1 v_2 \dots v_k$, we orient every $c$-half-edge incident with $v_i$ with $i$ odd inwards and the rest outwards (so that the directions alternate along the path), as shown on Figure~\ref{cubic2}. We repeat the process for all graphs $G_{12}$, $G_{13}$, $G_{23}$. Finally, we orient the remaining half-edges arbitrarily.
	\end{proof}

	\begin{thm}\label{meven}
	If $m$ is even and $G$ is a graph with $\Delta(G)=3$, then $\chi(\Gmn)=\omega(\Gmn)$.
	\end{thm}
	
	\begin{proof}
By Lemma~\ref{regular}, we can assume that $G$ is cubic. By Theorem~\ref{omegaB} we have that $\omega(\Gmn)=\frac{m}{2}\Delta(G) +1=\frac{3 m}{2}+1$. So, we only need to find a coloring of $\Gmn$ that uses $\frac{3 m}{2}+1$ colors. Also, we only need to consider $\Gmn$ with $n=m+2,\dots, 2m+1$ by Lemma~\ref{range}. We split the range $m+2,\dots, 2m+1$ into two parts.
	
	\textbf{Case 1}: $m+2\leq n\leq 3\frac{m}{2}+1$ (first range).
	Let $h$ be a good half-edge coloring of $G$, guaranteed by Lemma~\ref{lem1}, and let $\overrightarrow{G}$ be the orientation of $G$ given by Lemma~\ref{lem2} applied on $G$ and $h$. Then we define a vertex coloring $\varphi:V(\Gmn)\rightarrow \{0,1_1,\dots, 1_{\frac{m}{2}},2_1,\dots, 2_{\frac{m}{2}}, 3_1,\dots, 3_{\frac{m}{2}}\}$ as follows.
	
	\begin{enumerate}
	 \item \textbf{branch vertices:} $\varphi(v)=0$ for every branch vertex $v\in V(\Gmn)$.
	  \item \textbf{bubbles:}  Let $B_{uv}$ be a bubble of $\Gmn$. If the underlying half-edge $e_{uv}$ of $B_{uv}$ is oriented outwards, then let $\varphi(B_{uv})=h(e_{uv})$. Otherwise, let $\varphi(B_{uv})=\overline{h(e_{uv})}$.
	
	   \item \textbf{middle parts:} Let $k=|M_{uv}|$. Since $n$ is in the first range, $1\leq k \leq \frac{m}{2}$. Let $c\in [3]\setminus \{\varphi(B_{uv}), \varphi(B_{vu})\}$. If the $c$-half-edge oriented inwards is adjacent to $u$, then we define $\varphi(M_{uv})[1:k]=c[1:k]$ (as shown on Figure~\ref{cubic_first}). Otherwise we let $\varphi(M_{vu})[1:k]=c[1:k]$.
	\end{enumerate}
	
	 A routine check (similar as the one given in the proof of Theorem~\ref{miseven}) shows that vertices of the same color are of distance at least $m+1$ in $G^{\frac{1}{n}}$.
	
	\begin{figure}
	\begin{center}
	\subfloat[An orientation of half-edges of $G$.]{%
\label{cubic_first1}\includegraphics{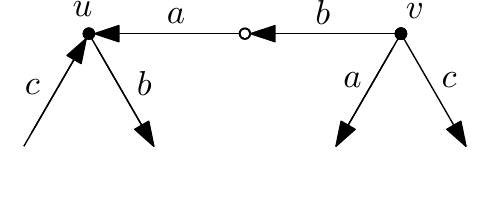}}
	\hspace{1cm}
	\subfloat[Corresponding coloring of $\Gmn$.]{%
\label{cubic_first2}\includegraphics{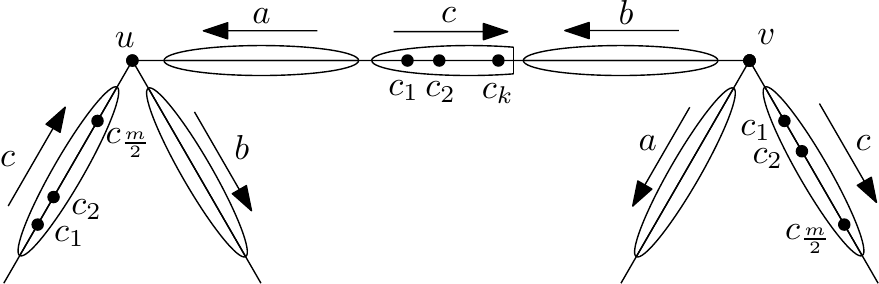}}
	\end{center}
	\caption{First range.}
	\label{cubic_first}
	\end{figure}
	
	\textbf{Case 2}: $3 \frac{m}{2}+2\leq n\leq 2m$ (second range).
	
 By Vizing's theorem, $G$ is $4$-edge-colorable. Let $g:E(G)\rightarrow \{1,2,3,*\}$ be a proper edge coloring with the smallest number of edges colored $*$. For $x\in \{1,2,3,*\}$, we call an edge that has color $x$ an \emph{$x$-edge}. For $a\in\{1,2,3\}$, let $G_{a}$ be the subgraph of $G$ induced by $a$-edges and those $*$-edges whose both endpoints are incident to an $a$-edge. Observe that $G_a$ is a disjoint union of paths and cycles. So, we can orient the edges in each component in the same direction (i.e. so that the indegree as well as the outdegree of each vertex is at most $1$ in $G_a$.)
Note that there is no conflict in orienting the edge since $G_{1}$, $G_{2}$, $G_3$ are edge-disjoint. Indeed, suppose that two of them, say $G_1$, $G_2$, share an edge $e$. Then $e$ is a $*$-edge, and by definition of $G_a$, both ends of $e$ are adjacent to one $1$-edge and one $2$-edge. But then we can color $e$ with the remaining color $3$, a contradicition with minimality of the number of $*$-edges.
On the other hand, every $*$-edge is oriented since it is adjacent to four edges, out of which two (non-adjacent) edges have to have the same color.

The construction of a proper vertex coloring $\varphi:V(\Gmn)\rightarrow \{0,1_1,\dots, 1_{\frac{m}{2}},\allowbreak 2_1,\dots, 2_{\frac{m}{2}}, \allowbreak 3_1,\dots, 3_{\frac{m}{2}}\}$ is more complicated that usual. Whereas both bubbles on an $a$-edge receive the same color-tuple $a$, the color-tuples on the two bubbles on a $*$-edge differ. Also, the coloring of middle parts for $1,2,3$-edges and $*$-edges i not be the same. In the first case we split each $M_{uv}$ into two (almost) equal parts, but in case of $*$-edges we split $M_{uv}$ into three non-equal parts.
The construction of $\varphi$ is depicted on Figure~\ref{F_cubic2}.

\begin{enumerate}
 \item \textbf{branch vertices:} $\varphi(v)=0$ for every branch vertex $v$.

\item \textbf{bubbles on $1,2,3$-edges}:

Let $uv\in E(G)$ that is oriented from $u$ to $v$ in $\overrightarrow{G}$. Then let $\varphi(B_{uv})=g(uv)$ and $\varphi(B_{vu})=\overline{g(uv)}$ (so the bubbles on one edge are oriented in the same direction).

\item \textbf{middle parts on $*$-edges}:

\noindent
\textbf{Case 1}: $|M_{uv}|-\frac{m}{2}$ is even.
Let $uv$ be a $*$-edge oriented from $u$ to $v$ in $\overrightarrow{G}$, and let $l=\frac{1}{2}(|M_{uv}|-\frac{m}{2})$. We split $M_{uv}$ into three parts $M_{uv}[1:l]$, $M_{uv}[l+1:l+\frac{m}{2}]$, and $M_{uv}[l+\frac{m}{2}+1: 2l+\frac{m}{2}]$.

Every $*$-edge is adjacent to at least one edge of each color $1$, $2$, $3$ (minimality). In particular, there is an $a$-edge adjacent to both $u$ and $v$ for some $a\in [3]$. Next, there is a $b$-edge and $c$-edge adjacent to $u$ and $v$, respectively, for $b,c\in [3]$ such that $a\neq b\neq c$.

First, we color the central part $M_{uv}[l+1:l+\frac{m}{2}]$ by letting $\varphi(M_{uv}[l+1:l+\frac{m}{2}])=a$.

Second, we color the remaining parts $M_{uv}[1:l]$ and $M_{uv}[l+\frac{m}{2}+1: 2l+\frac{m}{2}]$. Since $M_{uv}[l+\frac{m}{2}+1: 2l+\frac{m}{2}]=\overline{M_{vu}[1:l]}$, it suffices to show the coloring of $M_{uv}[1:l]$, the coloring of $M_{vu}[1:l]$ is analogous. If the $b$-edge adjacent to $u$ is oriented inwards, then we set $\varphi(M_{uv}[1:l])=b[1:l]$. Otherwise we define $\varphi(M_{uv}[1:l])=\overline{b}[1:l]$.
	
\noindent
\textbf{Case 2}: $|M_{uv}|-\frac{m}{2}$ is odd.	
We use the coloring defined in Case~1 on all but one vertex $(uv)_{m+1}$ of $P_{uv}$, and then we color $(uv)_{m+1}$ separately. So, $\varphi$ is defined as above, but for $l=\frac{1}{2}(|M_{uv}|-\frac{m}{2}-1)$ and the three parts $M_{uv}[1:l]$, $M_{uv}[l+1:l+\frac{m}{2}+1]\setminus (uv)_{m+1}$, and $M_{uv}[l+\frac{m}{2}+2: 2l+\frac{m}{2}+1]$.
It remains to color $(uv)_{m+1}$ on each $P_{uv}$.

 For any $a\in\{1,2,3\}$, an $a$-edge $uv$ is called a \emph{switching edge} if $uv$ is adjacent to $*$-edges $uu'$ and $vv'$ such that both $B_{u'u}$ and $B_{v'v}$ use the same color-tuple $a$. Then $H$ is a subgraph of $G$ induced by all switching edges and $*$-edges  adjacent to at least one switching edge. Every component $D$ of $H$ is either a path or an even cycle, so we can orient the edges in $D$ in the same direction. We refer to the resulting oriented graph obtained from $H$ by orienting each component as $\overrightarrow{H}$.

Now we let $\varphi((uv)_{m+1})=0$ for every $*$-edge $\overrightarrow{uv}\in E(\overrightarrow{H_a})$. The whole graph $G$ is colored, but we have a conflict on vertices colored with $0$. We do the following switching.  Let $\overrightarrow{uv}\in\overrightarrow{H}$ be an oriented edge such that $uv$ is an $a$-edge, where $a\in\{1,2,3\}$.  Then we let $\varphi((uv)_{\frac{m}{2}-2k})=0$ and $\varphi(u)=a_{\frac{m}{2}-2k}$.

\item \textbf{bubbles on $*$-edges}:

We adopt the notation from the previous part. Let $c\in [3]$ be the color missing around $u$. If the $c$-edge adjacent to $v$ is oriented inwards, then $\varphi(B_{uv}):=c$, otherwise $\varphi(B_{uv}):=\overline{c}$.

\item \textbf{middle parts on $1,2,3$-edges}:

We consider an $a$-edge $uv$, where $a\in [3]$. Let
$B_{uu'}$ and $B_{vv'}$ be bubbles adjacent to $u$ and $v$ that use different color-tuples. Then  $\varphi(M_{uv}[1:l]):=\varphi(\overline{B_{uu'}})[1:l]$, and $\varphi(M_{vu}[1:l]):=\varphi(\overline{B_{vv'}})[1:l]$.
\end{enumerate}

		\begin{figure}
	\begin{center}
	\subfloat[An orientation of $G$ with the edges colored by $g$.]{
\label{range1}\includegraphics[width=4cm]{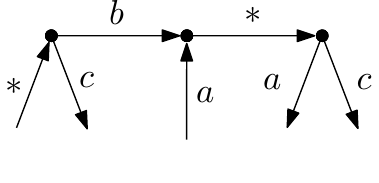}}
	\hspace{0cm}
	\subfloat[Coloring $\varphi$ of $\Gmn$. The numbers $2,3,4,5$ agree with the numbering of the steps in the proof.]{%
\label{range2}\includegraphics[width=\textwidth]{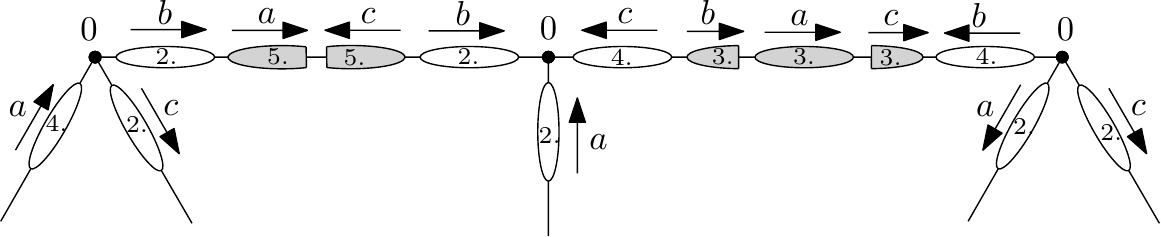}}
	\end{center}
	\caption{Second range.}
	\label{F_cubic2}
	\end{figure}	
			
	\end{proof}		
	\section{Counterexample and Theorem~2} \label{S_Odd}
	We begin this section by proving that Conjecture~\ref{conj1} is not true for $G=C_3 \Box P_2$ if $m=3$ and $n=5$. Before presenting the proof we need the following lemma.
	
	\begin{lem}\label{l_crust}
	Let $m$ be odd. If $\varphi$ is a proper vertex coloring of $\Gmn$ with $\omega(\Gmn)$ colors, then all vertices belonging to the same crust have the same color.
	\end{lem}
	
	\begin{proof}
	By Theorem~\ref{omegaB}, $\omega(\Gmn)= \frac{m-1}{2}\Delta+2$. By the proof of Theorem~\ref{omegaB} given in \cite{I}, maximal cliques in $\Gmn$ are induced by all vertices of distance $\frac{m-1}{2}$ and one vertex of distance $\frac{m+1}{2}$ from a branch vertex $v$ of degree $\Delta(G)$ in $G$. In our terminology, a maximal clique consists of one branch vertex $v$ with $d(v)=\Delta$, all bubbles at $v$ and one vertex $x$ of the crust $C_v$. The coloring $\varphi$ uses $\frac{m-1}{2}\Delta(G)+1$ colors on $v$ and the bubbles at $v$. Therefore, only one color remains for all the vertices of the crust $C_v$.
	\end{proof}
	
	By writing ``$\varphi(C_v)$'' for some coloring  $\varphi$ of $\Gmn$ we mean ``$\varphi(x)$ for all $x\in C_v$''.
	
	If $n\leq 2m+1$, $m$ is odd, and $\varphi$ is a proper vertex coloring of $\Gmn$ with $\omega(\Gmn)$ colors, then $\varphi(C_u)\neq \varphi(C_v)$ for every edge $uv\in E(G)$ since $|M_{uv}|\leq m-1$. So a necessary condition for $\varphi$ to be proper is that the vertex coloring $f$ of $G$ defined by $f(v):=\varphi(C_v)$ for every $v\in V(G)$ is proper.
	
	\begin{prop}\label{prop_prism}
	If $G=C_3 \Box K_2$, then $\chi(G^{\frac{3}{5}})>\omega(G^{\frac{3}{5}})$.
	\end{prop}
	
	\begin{proof}
	
	By Lemma~\ref{omegaB}, $\omega(\Gmn)=\frac{m-1}{2}\Delta(G)+2$. Thus, $\omega(G^{\frac{3}{5}})=5$. By way of contradiction, suppose there exists a proper vertex coloring $\varphi:V( G^{\frac{3}{5}})\rightarrow [5]$. Let $v_1,\dots,v_6$ be the vertices of $V(G)$ as on Figure~\ref{prism}. By Lemma~\ref{l_crust}, all vertices of the same crust have to receive the same color.
Since there are six crusts and only five colors, at least two crusts have to receive the same color, say $1$. Also, no two adjacent crusts can have the same color. So, we can assume that $\varphi(C_{v_1})=\varphi(C_{v_6})=1$ (all other cases are symmetric).

Every maximal clique has to use all five colors $1,2,3,4,5$. We try to color some vertex of each clique with $1$. The color $1$ cannot be used for any other crust then on $C_{v_1}$ and $C_{v_6}$ and thus, it has to be used on some vertex in distance at most $\frac{m-1}{2}=1$ from $v_k$ for every $k\in\{2,3,4,5\}$.

For $k=3$, the only vertex that can receive color $1$ is $(v_3 v_2)_1$, because all of the others are too close to one of the two crusts $C_{v_1}$, $C_{v_6}$.
	Analogously for $k=4$, the vertex $(v_4 v_5)_1$ has to receive color $1$. Subsequently $(v_5v_2)_1$ has to receive the color $1$, because $(v_5 v_6)_1$ and $v_5$ are too close to the crust $C(v_6)$ and $(v_5 v_4)_1$ is too close to $(v_4 v_5)_1$. But then there is no vertex  in distance at most $\frac{m-1}{2}$ from $v_2$ that can be colored with $1$. Indeed, $v_2$ and $(v_2 v_1)_1$ are too close to $C(v_1)$, $(v_2v_3)_1$ is too close to $(v_3 v_2)_1$ and $(v_2 v_5)_1$ is too close to $(v_5v_2)_1$. Hence, a contradiction.
	
	\end{proof}
	
	\begin{figure}
	\begin{center}
	\includegraphics{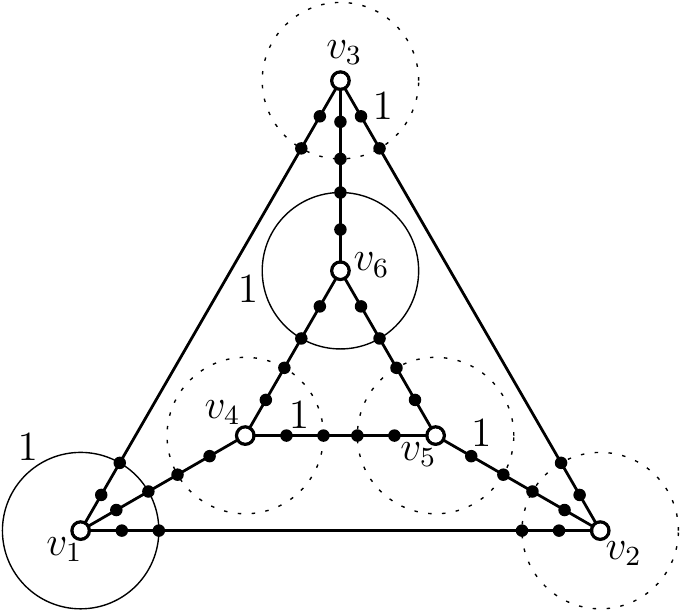}
	\caption{$(C_3\square K_2)^{\frac{3}{5}}$ is not $5$-colorable.}\label{prism}
	\end{center}
	\end{figure}
	
We now come to the two main theorems that support Conjecture~\ref{conj1} for odd $m$. We first find a general upper bound  $\omega(\Gmn)+2$ for the chromatic number of $\Gmn$ for any non-complete graph $G$ with maximum degree $\Delta\geq 4$. Then we show that in about half of the possible choices for $n$, the chromatic number of $\Gmn$ is in fact equal to $\omega(\Gmn)$. In the rest of this section, we assume that a graph $G$ is not a complete graph.

For a graph $G$ with $\Delta:=\Delta(G)$, let $f:V(G)\rightarrow [\Delta]$ be a proper vertex coloring and let $h:E(\Got)\rightarrow [\Delta]$ be a half-edge coloring of $G$. A half-edge $e_{uv}$ is called an \emph{incompatible half-edge} if $h(e_{uv})=f(v)$. If otherwise $h(e_{uv})\neq f(v)$, then $e_{uv}$ is called a \emph{compatible half-edge}. If every vertex of $V(G)$ is adjacent to at most $k$ incompatible half-edges, then we say that $f$ and $h$ are \emph{$k$-incompatible}. If $k=0$, then we say that that $f$ and $h$ are \emph{compatible} rather than $0$-incompatible.

\begin{lem}\label{compatible}
For every connected non-complete graph $G$ with maximum degree $\Delta$ and for every proper vertex coloring $f:V(G)\rightarrow [\Delta]$ there is a half-edge coloring $h:E(\Got)\rightarrow [\Delta]$ such that $f$ and $h$ are $2$-incompatible.
\end{lem}
	
\begin{proof}
By contradiction. We consider a half-edge coloring $h$ that has the minimal number of non-compatible half-edges. Let $u$ be a vertex with three neighbors $v_1,v_2,v_3$ such that $h(e_{uv_i})=f(v_i)$ for all $i=1,2,3$. Let $a:=h(e_{uv_1})=f(v_1)$, $b:=h(e_{uv_2})=f(v_2)$, $c:=h(e_{uv_3})=f(v_i)$. Observe that either $h(e_{v_1u})=b$ or $h(e_{v_2u})=a$, otherwise we could switch colors on $e_{uv_1}$ and $e_{uv_2}$, a contradiction with minimality. We can assume that $h(e_{v_1u})=b$. It follows that $h(e_{v_3u})=a$ (since otherwise we could switch colors on $e_{uv_1}$ and $e_{uv_3}$), and next that
$h(e_{v_2u})=c$ (since otherwise we could switch colors on $e_{uv_2}$ and $e_{uv_3}$). But now we can recolor $e_{uv_1}$ with $c$, $e_{uv_2}$ with $a$, $e_{uv_3}$ with $b$, a contradiction with minimality.
\end{proof}	
	
	\begin{thm}\label{odd2}
	If $m$ is odd and $G$ is a connected non-complete graph with maximum degree $\Delta\geq 4$, then $\chi(\Gmn)\leq\omega(\Gmn)+2$.
	\end{thm}
	
	\begin{proof}
	 Let $f:V(G)\rightarrow [\Delta]$ and $h:E(\Got)\rightarrow [\Delta]$, where $1=(1_1,\dots,1_{\frac{m-1}{2}}),\allowbreak \dots, \Delta=(\Delta_1,\dots,\Delta_{\frac{m-1}{2}})$, be $2$-compatible colorings given by Lemma~\ref{compatible}. Then we define an auxiliary (not necessarily proper) vertex coloring  $\varphi':V(\Gmn)\rightarrow \{0,\heartsuit, \diamondsuit,$ $1_1,\dots,1_{\frac{m-1}{2}},\dots,$ $\Delta_1,\dots,\Delta_{\frac{m-1}{2}}\}$.
	
	 \begin{enumerate}
	  \item \textbf{branch vertices}: $\varphi'(v)=0$  for every branch vertex $v\in\Gmn$.
	  \item \textbf{crusts}: $\varphi'(C_v)=f(v)[\frac{m-1}{2}]$ for every branch vertex $v\in\Gmn$
	  \item \textbf{bubbles}:	
	  $$\varphi'(B_{uv})=\left\{
	  \begin{array}{llllllll}
	  (\heartsuit)*h(e_{uv})[1:\frac{m-1}{2}-1] & \mbox{ if }  h(e_{uv})=f(u),\\
	   h(e_{uv})[1:\frac{m-1}{2}-1]*(\diamondsuit)  & \mbox{ if } h(e_{uv})=f(v),\\
	   h(e_{uv}) &\mbox{ otherwise.}
	  \end{array}
	  \right.
	 $$
	\end{enumerate}
	
Observe that there can only be a conflict between vertices using color $\heartsuit$ and vertices using color $\diamondsuit$.
Vertices colored $\heartsuit$ are exactly vertices $(uv)_1$ for which $h(e_{uv})=f(u)$, and we will call them $\heartsuit$-vertices. If $n=m+2$, then for every superedge $P_{uv}$, the vertices $(uv)_1$ and $(vu)_1$ are of distance $m$ in $\Got$. So, if both $(uv)_1$, $(vu)_1$ are $\heartsuit$-vertices, then we have a conflict. For each superedge containing two $\heartsuit$-vertices, we choose one of the two $\heartsuit$-vertices. We call the set of all the chosen $\heartsuit$-vertices as a $(\heartsuit \rightarrow 0)$-set, and the set of their neighboring branch vertices as a $(0 \rightarrow \heartsuit)$-set.


Vertices colored $\diamondsuit$ are vertices $(uv)_{(m-1)/2}$ for which $h(e_{uv})=f(v)$, and we will call them $\diamondsuit$-vertices. Since $f$ and $h$ are $2$-incompatible, we have at most two $\diamondsuit$-vertices around each branch vertex. The subgraph $H$ of $\Gmn$ induced by all $\diamondsuit$-vertices then consists of disjoint union of paths and even cycles. So, we can properly color $H$ using two colors $\diamondsuit_1$ and $\diamondsuit_2$. Let $\diamondsuit_1$-set and $\diamondsuit_2$-set denote the set of $\diamondsuit$-vertices colored
$\diamondsuit_1$ and $\diamondsuit_2$, respectively.

We now define $\varphi:V(\Gmn)\rightarrow \{0, \heartsuit, \diamondsuit_1, \diamondsuit_2, 1_1,\dots,1_{\frac{m-1}{2}},\dots, \Delta_1,\dots,\Delta_{\frac{m-1}{2}}\}$, for branch vertices, crusts, and bubbles.

$$\varphi(v)=\left\{
\begin{array}{lllllllll}
0 & \mbox{ if }  v\in  (\heartsuit \rightarrow 0)\mbox{-set},\\
\heartsuit & \mbox{ if } v\in (0\rightarrow \heartsuit)\mbox{-set},\\
\diamondsuit_1& \mbox{ if }  v\in \diamondsuit_1\mbox{-set},\\
\diamondsuit_2& \mbox{ if }  v\in \diamondsuit_2\mbox{-set},\\
\varphi'(v)& \mbox{ otherwise (if } v \mbox{ is not from a middle part).}
\end{array}
\right.
$$
\begin{enumerate}
	\item[4.] \textbf{middle parts}:
 Let $k=\lceil \frac{|M_{uv}|}{2} \rceil$ and let $l=\lfloor \frac{|M_{uv}|}{2} \rfloor$.  Let $a,b \in[\Delta]$ such that $\varphi(B_{uv})=a$, $\varphi(B_{vu})=b$. Then we have the following three (up to symmetry) cases.

	 \begin{enumerate}
	 \item $\varphi(C_{u}),\varphi(C_{v})\in\{a_{\frac{m-1}{2}},b_{\frac{m-1}{2}}\}$. Then we fix $c,d\in [\Delta] \setminus\{a,b\}$ ($c\neq d$ exist since $\Delta\ge 4$) and  define $\varphi(M_{uv}[1:k])\allowbreak =\overline{c}[1:k]$ and $\varphi(M_{vu}[1:l])=\overline{d}[1:l].$
	 \item $\varphi(C_{u})=c_{\frac{m-1}{2}}$, $\varphi(C_{v})=a_{\frac{m-1}{2}}$, where $c\in[\Delta]\setminus\{a,b\}$. Then we  fix $d\in [\Delta] \setminus\{a,b,c\}$ and define $\varphi(M_{uv}[1:k])=\overline{c}[2:k]*(\heartsuit)$ and $\varphi(M_{vu}[1:l])=\overline{d}[1:l].$
	 	 \item $\varphi(C_{u})=c_{\frac{m-1}{2}}$, $\varphi(C_{v})=d_{\frac{m-1}{2}}$, where $c,d\in [\Delta]\setminus\{a,b\}$. Then we define $\varphi(M_{uv}[1:k])=\overline{c}[2:k]*(\heartsuit)$ and $\varphi(M_{vu}[1:l])=\overline{d}[2:l]*(\diamondsuit_1).$
	 \end{enumerate}
	 \end{enumerate}

		\end{proof}

\begin{cor}[of the proof of Theorem~\ref{odd2}]\label{compatible_cor}
If $m$ is odd and $G$ is a non-complete graph with maximum degree $\Delta\geq 5$ and such that there are compatible proper colorings $f:V(G)\rightarrow [\Delta]$ and $h:E(\Got)\rightarrow [\Delta]$, then $\chi(\Gmn)=\omega(\Gmn)$.
\end{cor}
	
\begin{proof}
If $f$ and $h$ are compatible, then there are no
$\diamondsuit$-vertices in $\Gmn$, so we do not need to use colors $\diamondsuit_1$ and $\diamondsuit_2$. The only other place where we use one of these colors is the middle part, in Case (c), where we use the color $\diamondsuit_1$. But if $\Delta\geq 5$, then there exists a fifth color $e\in [\Delta]\setminus \{a,b,c,d\}$. So, we can define $\varphi(M_{uv}[1:k])=\overline{c}[2:k]*(\heartsuit)$ and $\varphi(M_{vu}[1:k])=\overline{e}[1:k].$
\end{proof}

When each middle part has at least $\frac{m+1}{2}$ and at most $m-1$ vertices, the situation gets easier, as the following theorem shows. (Note that $n=2|B_{uv}|+2+|M_{uv}|+1$.)
	
	\begin{thm}
	Let $m$ be odd and $G$ be a non-complete graph with $\Delta(G)\geq 3$. If $\frac{3m+5}{2}\leq n\leq 2m$, then $\chi(\Gmn)=\omega(\Gmn)$. If $n=2m+1$, then $\chi(\Gmn)\le \omega(\Gmn)+1$.
	\end{thm}
	
	\begin{proof}
	
	Let $\Delta=\Delta(G)$. By Lemma~\ref{regular}, we can assume that $G$ is $\Delta$-regular. We start with the first part of the claim, so we suppose $\frac{3m+5}{2}\leq n\leq 2m$. Since $\omega(\Gmn)=\frac{m-1}{2}\Delta+2$ and $\omega(\Gmn)\leq \chi(\Gmn)$, we only need to show a coloring of $\Gmn$ that uses $\frac{m-1}{2}\Delta+2$ colors.
	
	Let $f:V(G)\rightarrow [\Delta]$ be a proper vertex coloring of $G$ and let $h:E(\Got)\rightarrow [\Delta]$ be a proper half-edge coloring of $G$. Let $uv$ be an edge of $G$. If $f(u)=h(e_{uv})$, then the vertex $(u,v)_1$ is called a $\heartsuit$-vertex. For each superedge containing two $\heartsuit$-vertices, we choose one and put it in the (originaly empty) set called $(\heartsuit \rightarrow 0)$-set. The set of branch vertices that have a neighbor in the $(\heartsuit \rightarrow 0)$-set is called $(0 \rightarrow \heartsuit)$-set. We define a coloring 	 $\varphi:V(\Gmn)\rightarrow \{0, \heartsuit,1_1,\dots,1_{\frac{m-1}{2}},\dots, \Delta_1,\dots,\Delta_{\frac{m-1}{2}},\}$ as follows.
	\begin{enumerate}
	 \item \textbf{branch vertices:} $\varphi(v)=\heartsuit$ if $v\in (0\rightarrow \heartsuit)$-set, $\varphi(v)=0$ otherwise.
	 \item \textbf{$\heartsuit$-vertices:}  $\varphi(v)=0$ if $v\in  (\heartsuit\rightarrow 0)$-set, $\varphi(v)=\heartsuit$ otherwise.
	 \item \textbf{crusts:}  $\varphi(C_v)=f(v)[1]$.
	 \item \textbf{bubbles:} If there is no $\heartsuit$-vertex in $B_{uv}$, then $\varphi(B_{uv})=h(e_{uv})$. Otherwise $\varphi(B_{uv}[2:\frac{m-1}{2}])=h(e_{uv})[2:\frac{m-1}{2}]$.
	 \item \textbf{middle parts} Let $k=\lceil \frac{1}{2}|M_{uv}| \rceil$ and let $l=\lfloor \frac{1}{2}|M_{uv}| \rfloor$. If $k=\frac{m-1}{2}$ and $h(e_{uv}), h(e_{vu}), f(u), f(v)$ are pairwise distinct (so there is no $\heartsuit$-vertex on $P_{uv}$), then we define $\varphi(M_{uv}[1:k])=\overline{f(v)}[1:k-1]*(\heartsuit)$. In all other cases we define $f(M_{uv}[1:k])=\overline{f(v)}[1:k]$. Lastly, we define $f(M_{vu}[1:l])=\overline{f(u)}[1:k]$.

	\end{enumerate}
 The second part of the theorem follows immediately by inserting one more vertex $(uv)_*$ on every superedge of $G^{\frac{m}{2m}}$ and coloring it with a new color. In particular, to obtain $G^{\frac{m}{2m+1}}$ from $G^{\frac{m}{2m}}$, we replace the edge $M_{uv}[k] M_{vu}[l]$ by a path  $M_{uv}[k] (uv)_* M_{vu}[l]$  for every edge $uv\in E(G)$, and color all the new vertices $(uv)_*$ by a new color.
 \end{proof}
	
	\section{Theorem 2 for complete graphs}\label{S_Complete}
	
\begin{lem}\label{comp1}
$\chi(G^{\frac{3}{5}})=\omega(G^{\frac{3}{5}})$ for any complete graph $G$ with $\Delta(G)\geq 3$.
\end{lem}
	
\begin{proof}
We show that we can color $G^{\frac{3}{5}}$ with $\omega(G^{\frac{3}{5}})=\Delta(G)+2$.
If $\Delta(G)=3$, then $G=K_4$, for which we give a coloring on Figure~\ref{K4}. Let $v_1,\dots,v_r$ be vertices of $G$. For a fixed $\Delta(G)\geq 4$, we construct a coloring $\varphi:V(G^{\frac{3}{5}})\rightarrow \{0,\dots,\Delta+1\}$ as follows. First, we color   all branch vertices of $G^{\frac{3}{5}}$ with $0$. Second, we let $\varphi(v)=i$ for every vertex of the same crusts $C_{v_i}$. Finally, we need to color neighbors of $v_1,\dots,v_{r}$.

Observe that all the vertices of distance $3$ from a branch vertex $v_i$ have mutually different colors $1,\dots, i-1, i+2, \dots, \Delta+1$, which have to be used for the neighbors of $v_i$.


Suppose that we colored all the neighbors of $v_1, \dots, v_{i-1}$ and we want to color $v_i$. We construct an auxiliary bipartite graph $(A_i,B_i)$ with parts $A_i:=\{(v_iv_j)_1: j\in [\Delta+1]\setminus\{i\}\}$ and  $B_i:=[\Delta+1]\setminus\{i\}$.
 There is an edge between $(v_iv_j)_1\in A_i$ and $b\in B_i$ if and only if $b$ is not forbidden for $(v_iv_j)_1$, i.e. if $b$ can be used on $(v_iv_j)_1$ such that the resulting coloring remains proper. Each vertex $(v_iv_j)_1$ has one or two forbidden colors: the color $\varphi((v_iv_j)_3)=j$, and in case  $j<i$ also the color $\varphi((v_iv_j)_4)$. We show that either $(A_i,B_i)$ has a perfect matching or that we can switch colors some vertices and redefine $(A_i,B_i)$ so that the new $(A_i,B_i)$ has a perfect matching. By  Hall's Theorem [ref!!!], $(A_i,B_i)$ has a perfect matching if and only if $|S|\leq |N(S)|$ for all $S\subseteq A_i$.

 Since every vertex of $A_i$ has degree at least $\Delta-2$ in $(A_i,B_i)$, $|N(S)|\geq \Delta-2$, so we $S$ can only violate the Hall's condition if $|S|=\Delta-1$ or $|S|=\Delta$.

Suppose first that $|S|=\Delta-1$. Since $\Delta\geq 4$, there are three vertices $(v_i v_j)_1$, $(v_i v_k)_1$, $(v_i v_l)_1$ in $S$. If $|N(S)|=\Delta-2$, then each of these vertices has the same two forbidden colors $b_1,b_2 \in B_i$. But this is not possible since the forbidden colors $j,k,l$ for $(v_i v_j)_1$, $(v_i v_k)_1$, $(v_i v_l)_1$, respectively, are pairwise distinct.

Let now $|S|=\Delta$, so $S=A_i$.
We claim that if $i\in [r-2]$, then $N(S)=B_i$.  Indeed, since $r-1$ and $r$ are the only forbidden colors for $(v_iv_{r-1})_1$ and $(v_iv_r)_1$, respectively, and these are distinct, we conclude that $N(\{(v_i v_{r-1})_1,(v_iv_r)_1\})=B_i$ and thus $N(S)=B_i$.

Let now $i=r-1$. If a perfect matching does not exist in $(A_i,B_i)$, then $|N(A_i)|=\Delta-1$. Since the only forbidden color for $(v_i v_{r})_1$ is $r$, we have $\varphi((v_iv_j)_4)=r$, or equivalently $\varphi((v_jv_i)_1)=r$, for all $j\in [\Delta]\setminus \{i\}$. Let $k\in [\Delta]\setminus \{i,r-1,r\}$ such that $\varphi((v_1v_k)_1)\neq r-1$. Then
we switch colors on vertices $(v_1v_i)_1$ and $(v_1v_k)_1$ (see Figure~\ref{complete1}). For the redefined bipartite graph $(A_i,B_i)$ we then have $|N(A_i)|=|A_i|$, which means that we can find a perfect matching in $(A_i,B_i)$.

Finally, let $i=r$, and suppose $N(A_i)=\Delta-1$. Let $b\in B_i\setminus  N(A_i)$, so $b$ is forbidden for every vertex in $A_i$. Exactly one vertex $(v_i v_b)_3$ of $\{(v_i v_j)_3:j=[\Delta]\setminus \{i\}\}$ has color $b$. So, all vertices of $\{(v_i v_j)_4:j=[\Delta]\setminus \{i,b\}\}$ have color $b$. Let $l\in [\Delta]\setminus \{i,b,r\}$ such that $\varphi((v_1v_l)_1)\neq r$. Then
we switch colors on vertices $(v_1v_i)_1$ and $(v_1v_l)_1$ to obtain a new graph $(A_i,B_i)$ with $|N(A_i)|=|A_i|$ as needed. See Figure~\ref{complete2}.
\end{proof}
\begin{figure}[ht]
  \begin{center}
 \subfloat[$5$-coloring of $(K_4)^{\frac{3}{5}}$.]{\label{K4}
\includegraphics[width=4.7cm]{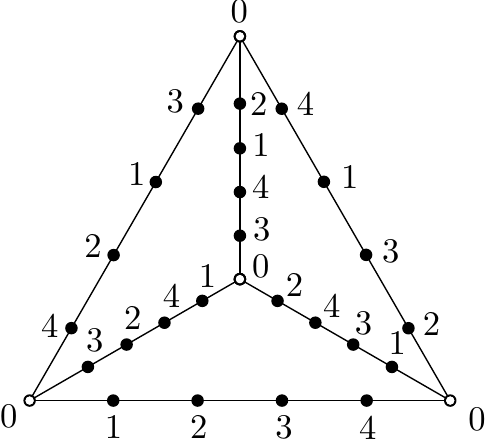}}
  \hspace{0cm}
  \subfloat[$i=r-1$]{\label{complete1}
    \includegraphics{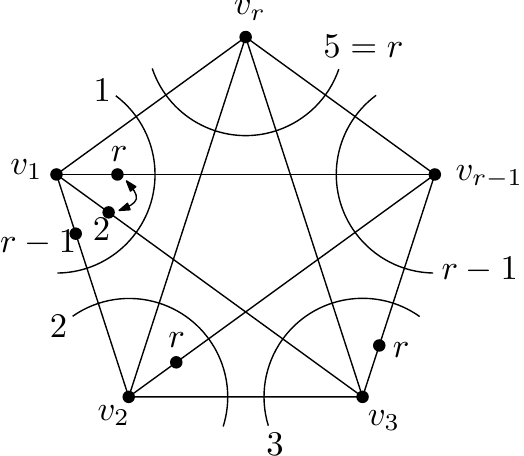}}
     \hspace{0cm}
      \subfloat[$i=r$]{\label{complete2}
    \includegraphics{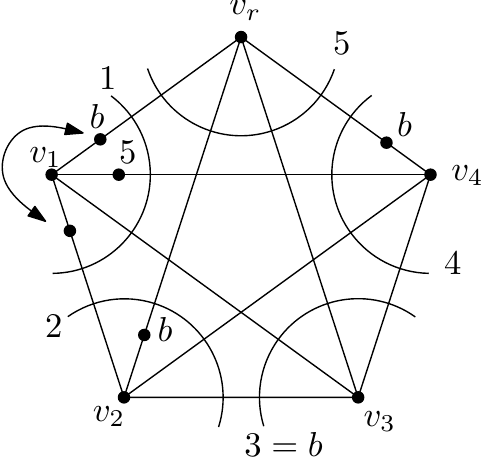}}
      \caption{}\label{fig_complete}
  \end{center}
\end{figure}
\begin{thm}
If $G$ is a complete graph, then $\chi(\Gmn)=\omega(\Gmn)$.
\end{thm}

\begin{proof}

We only need to consider the case when $m$ is odd since Theorem~\ref{miseven} and Theorem~\ref{meven} can both be applied to complete graphs.

We first show that for all $m\geq 3$, there exists a proper coloring of $\varphi: V(G^{\frac{m}{m+2}})\rightarrow \{0,\dots,\frac{m-1}{2}\Delta+1\}$ such that the vertices colored $0$ are exactly the branch vertices. We use induction on $m$. The base case is exactly Lemma~\ref{comp1}.

Let now $m\geq 5$. By the induction hypothesis, there exists a proper vertex coloring $\varphi':V(G^{\frac{m-2}{m}})\rightarrow \{0,\dots,\frac{m-3}{2}\Delta+1\}$. Let $f:V(G^{\frac{2}{3}})\rightarrow \{0, \frac{m-3}{2}\Delta+2,\dots, \frac{m-1}{2}\Delta+1\}$ be a proper coloring such that $0$ is used exactly on branch vertices (existing by Lemma~\ref{G23}). Then we define
$$\varphi((uv)_i)=\left\{
\begin{array}{lllllll}
f((uv)_i) &\mbox{ in } G^{\frac{2}{3}} & \mbox{ if } i\in \{0,1\},\\
\varphi'((uv)_{i-1}) & \mbox{ in } G^{\frac{m-2}{m}} & \mbox{ otherwise.}
\end{array}\right.
$$

It is now easy to complete the proof. Suppose that we have to construct a coloring of $\Gmn$ for some $n>m+2$ that uses $\frac{m-1}{2}\Delta+2$ colors. Then we adopt the above coloring $\varphi$ for all $(uv)_i$ with $i=0,\dots, \frac{m+1}{2}$, so $\varphi$ provides a coloring of bubbles and crusts in $\Gmn$. All we need to do now is to color the middle parts.
By Lemma~\ref{range}, we can assume that $m+2\leq n\leq 2m+1$. This means that $0\leq |M_{uv}|\leq m-1$. For $i=1,\dots, \lceil \frac{1}{2}|M_{uv}|\rceil $, let $S_i$ be the set of colors used on vertices of distance $i$ from $u$ or from $v$ in $G^{\frac{1}{n}}$. By construction of $\varphi$ we have $S_i\cap S_j=\emptyset$ if $i\neq j$. Therefore, for each $P_{uv}$ there exist at least $2$ colors $c_1$, $c_2$ in $S_i$ that are not used on $P_{uv}$.  Then we define $\varphi(M_{uv}[i])=c_1$ and $\varphi(M_{vu}[i])=c_2$ (if $M_{vu}[i]$ exists).

\end{proof}

	\section{Dynamic coloring}\label{S_dynamic}

In this section we assume that $G$ is a graph with maximum degree $\Delta\geq 4$ and that $m$ is odd. Recall from Section~\ref{S_Odd} that if there exist compatible colorings $f:V(G)\rightarrow [\Delta]$ and $h: E(\Got)\rightarrow[\Delta]$ and $\Delta(G)\geq 5$, then $\chi(\Gmn)=\omega(\Gmn)$ (Corollary~\ref{compatible_cor}).
The purpose of this section is to find a sufficient condition for existence of compatible colorings $f$ and $h$ for $G$. An \emph{$r$-dynamic proper $k$-coloring} of a graph $G$ is a proper
coloring $g:V(G)\rightarrow[k]$ such that for every vertex $v\in V(G)$ the number of colors used on $N(v)$ is at least $\min\{r,d(v)\}$.

\begin{lem}\label{lem_dyn}
Let $G$ be a graph with maximum degree $\Delta\geq 4$. For every $4$-dynamic proper $\Delta$-coloring $f$ of $G$ there exists a  coloring $h:E(\Got)\rightarrow [\Delta]$ such that $f$ and $h$ are compatible.
\end{lem}

\begin{proof}
We subsequently color half-edges around each vertex, in any fixed order of vertices of $G$. Suppose that we want to color the half-edges around $u$. We construct an auxiliary bipartite graph $(A,B)$ with parts $A:=\{e_{u v}:v\in N(u)\}$ and  $B:=[\Delta+1]\setminus\{f(u)\}$. There is an edge between $e_{uv}\in A$ and $b\in B$ if and only if $e_{uv}$ can be colored with $b$, i.e. exactly when $b\neq f(u)$ and $b\neq h(e_{vu})$ (in case $e_{vu}$ is already colored). We show that either $(A,B)$ has a perfect matching or that we can switch colors on some half-edges and redefine $(A,B)$ so that the new $(A,B)$ has a perfect matching. If $|N(S)|\geq |S|$ for all $S\subseteq A$, then the existence of a perfect matching is ensured by Hall's theorem. So suppose that there exists $S\subseteq A$ such that $|N(S)|<|S|$. Every vertex of $A$ has degree at least $\Delta-2$, where $\Delta\geq 4$. So $|N(S)|\geq \Delta-2$ and therefore $|S|\geq \Delta-1$.

First, suppose that $|S|=\Delta-1$. Let $e_{uv'}\in A\setminus S$. Since $f$ is $4$-dynamic, there are at least four colors on $N_G(v)$. This implies that there are three vertices $v_1$, $v_2$, $v_3$ in $N_G(v)\setminus \{v'\}$ with pairwise different colors. If $|N(S)|\leq \Delta-2$, then there are $b_1,b_2\in B\setminus N(S)$. But then one of the vertices $v_1$, $v_2$, $v_3$ is colored neither $b_1$ nor $b_2$. Therefore one of $b_1,b_2$ is not forbidden for that vertex, a contradiction.

Second, suppose that $|S|=\Delta$, i.e. $|A|>|N(A)|$. Since then $|N(A)|=\Delta+1$, there is a color $b\in B_i$ that is forbidden for every vertex of $A_i$. So, for every $e_{uv}\in A$, either $f(v)=b$ or $h(e_{vu})=b$ (if defined). Since $f$ is $4$-dynamic, there exists $x_1\in N_G(v)$ such that $f(x_1)\neq b$. Then necessarily $h(e_{x_1u})= b$.  Let now $x_2\in N_G(x_1)\setminus\{u\}$ such that $f(x_2)\neq b$ and $h(e_{x_1x_2})\neq f(u)$. Such a vertex exists since $f$ is $4$-dynamic. If $h(e_{x_2x_1})\neq b$, then we can switch the colors on $e_{x_1 u}$ and $e_{x_1 x_2}$. Then the new graph $(A,B)$ has a perfect matching and the so far defined coloring $h$ is proper and compatible with $f$. If otherwise $h(e_{x_2x_1})=b$, then we search for $x_3\in N_G(x_2)\setminus\{x_1\}$ such that $f(x_3)\neq b$ and $h(e_{x_2x_3})\neq f(x_1)$, and check whether or not $h(e_{x_3x_2})\neq b$. We repeat this process until we find a vertex $x_k$ with $h(e_{x_kx_{k-1}})\neq b$. Then we switch colors on $e_{x_i x_{i-1}}$ and $e_{x_ix_{i+1}}$ for all $i=1,\dots,k-1$, where $x_0=u$.
The process has finitely many steps since $|V(G)|$ is finite and we cannot have a loop. Indeed, suppose that we found a sequence $x_1, x_2, \dots, x_i, x_{i+1}, \dots, x_{j}, x_i,\dots$. Then $h(e_{x_{i-1}x_i})=b$ and $h(e_{x_jx_i})=b$ since otherwise the sequence would end at $x_{i-1}$ and $x_j$, respectively. But $e_{x_{i-1}x_i}$ and $e_{x_jx_i}$ are adjacent and cannot have the same color, a contradiction.
\end{proof}
	
	\begin{figure}[ht]
  \begin{center}
  \subfloat[Switching along the path $u x_1 x_2 \dots x_k$.]{\label{dynamic1}
    \includegraphics{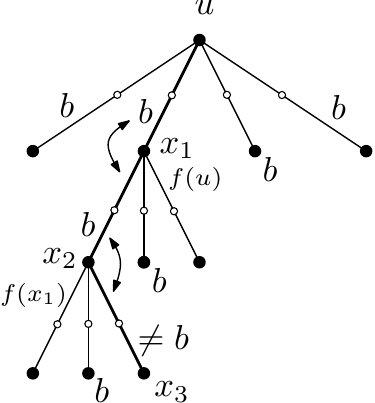}}
     \hspace{3cm}
      \subfloat[There is no loop.]{\label{dynamic2}
    \includegraphics{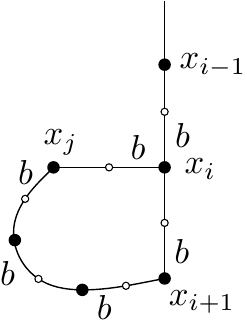}}
      \caption{Coloring half-edges around $u$.}\label{dynamic}
  \end{center}
\end{figure}	
\begin{rem}
Lemma~\ref{lem_dyn} holds under a weaker assumption that $f$ is $3$-dynamic and such that for every vertex $v$ with exactly three colors used on $N(v)$, at least two colors are used twice on $N(v)$.
\end{rem}	

The following Theorem by Jahanbekam, Kim, O, and West~\cite{Sogol}, \cite[Corollary 5.2.6]{O} provides sufficient conditions for a graph to be $r$-dynamic proper $k$-colorable.
\begin{thm}
 If G is a $\Delta$-regular graph with $\Delta\geq 7r  \ln r$, then there exists $r$-dynamic proper $(r\chi(G))$-coloring of $G$.
\end{thm}

We seek $4$-dynamic proper $\Delta$-coloring of $G$. For that we need $\Delta\geq \lceil 28 \ln 4 \rceil =39$ and $\chi(G)\leq \frac{\Delta}{4}$. The following corollary is immediate.

\begin{cor}
If $G$ is a $\Delta$-regular graph with $\Delta\geq 39$ and $\chi(G)\leq \frac{\Delta}{4}$, then $\chi(\Gmn)=\omega(\Gmn)$.
\end{cor}

\end{document}